\documentclass{article}

\usepackage{amsmath}
\usepackage{amsthm}
\usepackage{amssymb}
\usepackage{mathrsfs}
\usepackage{graphicx}
\usepackage{verbatim}
\usepackage{tikz}
\usepackage[all]{xy}
\usepackage{graphics}
\usepackage[colorlinks=true]{hyperref}
\hypersetup{urlcolor=blue, citecolor=red}

\newtheorem{theorem}{Theorem}[section]

\theoremstyle{definition}%da qui in poi negli environment introdotti il testo non � in italico 

\newtheorem{remark}{Remark}

\newcommand{\R}{\mathbb{R}}

\def\beq{\begin{equation}}
\def\eeq{\end{equation}}

\def\f{\varphi}

\title{\sc Symmetry of components, Liouville-type theorems and classification results for some nonlinear elliptic systems}
\date{}
\author{Alberto Farina }
\begin{document}
\numberwithin{equation}{section}
%%%%%%%%%%%%%%%%%%%%%%%%%%%%%%%%%%%%%%%%%%%%%%%%%%%%%%%%%%%%%%%%%%%%%%%%%%%%%%%
\maketitle
%%%%%%%%%%%%%%%%%%%%%%%%%%%%%%%%%%%%%%%%%%%%%%%%%%%%%%%%%%%%%%%%%%%%%%%%%%%%%%%
{\footnotesize
%\centerline {Alberto FARINA }
\centerline{LAMFA, CNRS UMR 7352, Universit\'e de Picardie Jules Verne}
\centerline{33 rue Saint-Leu, 80039 Amiens, France}
\centerline{and}
\centerline{Institut Camille Jordan, CNRS UMR 5208, Universit\'e Claude Bernard Lyon I}
\centerline{43 boulevard du 11 novembre 1918, 69622 Villeurbane cedex, France}
\centerline{email: alberto.farina@u-picardie.fr}}

\centerline{}

\begin{abstract}
We prove the symmetry of components and some Liouville-type theorems for, possibly sign changing, entire distributional solutions  to a family of  nonlinear elliptic systems encompassing models arising in Bose-Einstein condensation and in nonlinear optics. For these models we also provide precise classification results for non-negative solutions. The sharpness of our results is also discussed. 
\end{abstract}

\smallskip

%\noindent \textbf{Keywords:} {\footnotesize }
\quad \textbf {} {\footnotesize }

\section{Introduction}

The aim of this paper is to classify entire solutions $(u,v)$ of the following nonlinear elliptic systems arising in Bose-Einstein condensation and in nonlinear optics (see for instance \cite{F}, \cite{KL-D}, \cite{Perez-Beitia} and the references therein) : 

\beq\label{Sistema}
\begin{cases}

- \Delta u  +\alpha \vert u \vert^{m-1} u  = - \lambda \vert u \vert^{\theta-1} u   + \beta \vert u \vert^s 
\vert v \vert^{s+\gamma -1} v  & \text{in $\R^N$}\\
- \Delta v  +\alpha \vert v \vert^{m-1} v  = - \lambda \vert v \vert^{\theta-1} v   + \beta \vert v \vert^s 
\vert u \vert^{s+\gamma -1} u & \text{in $\R^N$}
\end{cases}
\eeq
where $\alpha $, $\beta$ and $ \lambda $ are {\it {continuous}} functions defined on $\R^N$, $N\ge1$. \\

The present paper is motivated by the recent and interesting work \cite{QS1}, as well as by the stimulating discussions with the Authors of \cite{QS1}. 

More precisely, we shall prove that any entire distributional solution $(u,v)$ of system \eqref{Sistema}, has the symmetry property $u = v$ ({\it symmetry of components}). We use this result to establish some new Liouville-type theorems as well as some classification results. 

Our method is different (and complementary) from the one used in \cite{QS1}. It exploits the attractive character of the interaction between the two states $u$ and $v$.  It applies to any distributional entire solution, possibly sign-changing and without any other restriction.  Also, it applies to systems with nonlinearities that are not necessarily positive (or cooperative) nor necessarily homogeneous. 
Section 2 is devoted to the main results, while in section 3 we consider their extension to more general models involving nonlinearities which are not necessarily of polynomial type. We also discuss the sharpness of our results. 

\section{Model problems and main results}

Throughout the section $\alpha $, $\beta$ and $ \lambda $ will be {\it {continuous}} functions defined on $\R^N$, $N\ge1$. \\

\begin{theorem} {\bf (Symmetry of components)} \label{teo1}
Let $ N\ge 1 $ and assume $ m>0$, $\theta >1$, $s\ge0$, $\gamma \ge1$, $\alpha = \alpha(x) \ge0$, $\beta = \beta(x) \ge0$, $ \lambda = \lambda(x) \ge \lambda_0>0$. 

Let $(u,v)$ be a distributional solution of system \eqref{Sistema} such that  $ u, v \in L^p_{loc}(\R^N)$,  with $ p = \max \{m, \theta, 2s+\gamma\}$. 

Then $u = v$.  
\end{theorem}

\begin{proof} From \eqref{Sistema} we have

\beq \nonumber
\Delta (u-v) = \alpha (\vert u \vert^{m-1} u - \vert v \vert^{m-1} v ) + 
\eeq 
\beq \label{differenzaeq} 
+ \lambda ( \vert u \vert^{\theta-1} u - \vert v \vert^{\theta-1} v ) + \beta \vert u \vert^s \vert v \vert^s (\vert u \vert^{\gamma-1} u  - \vert v \vert^{\gamma-1} v ) \qquad {\rm in} \quad {\cal D}^{'} (\R^N).
\eeq

Set $\psi = u-v$. The assumptions on $u$ and $v$ imply that $\psi \in L^p_{loc}(\R^N)$ and $ \Delta \psi$ belongs to $ L^1_{loc}(\R^N)$. Thus we can apply Kato's inequality \cite{Kato, Brezis} to get  

\beq \label{differenzaeqKato} 
\Delta (\psi^+) \ge \Delta \psi 1_{\{ u-v>0\}} \ge  \lambda_0 ( \vert u \vert^{\theta-1} u - \vert v \vert^{\theta-1} v ) 1_{\{ u-v>0\}} \qquad {\rm in} \quad {\cal D}^{'} (\R^N),
\eeq
where $ 1_E$  denotes the characteristic function of the measurable subset $ E \subset \R^N$. 

\smallskip

Reminding the well known inequality
\beq \vert t \vert^{q-1} t - \vert s \vert^{q-1} s \ge c_q (t-s)^q, \quad  for\ \  t>s\qquad(q\ge1),\label{wki}\eeq
from \eqref{differenzaeqKato}  we obtain

\beq \label{disugBrezis} 
\Delta (\psi^+) \ge \lambda_0  c_q (\psi^+)^{\theta}  \qquad {\rm in} \quad {\cal D}^{'} (\R^N).
\eeq
 Since $\lambda_0  c_q >0$ and $\theta >1$ we immediately get $\psi^+ =0$ (cfr. Lemma 2 of \cite{Brezis}).  Hence $ u \le v$ a.e. on $\R^N.$ Finally, exchanging the role of $u$ and $v$ we obtain the desired conclusion $u = v$.\end{proof}

\medskip

 Then we are in position to prove the following 
 
 \medskip

\begin{theorem} {\bf (of Liouville-type)} \label{teo2}
Assume $ N\ge 1 $ and let $(u,v)$ be a distributional solution of system \eqref{Sistema},  where $ m>0$, $\theta >1$, $s\ge0$, $\gamma \ge1$, $\alpha = \alpha(x) \ge0$, $\beta = \beta(x) \ge0$, $ \lambda = \lambda(x) \ge \lambda_0>0$ and $ u, v \in L^p_{loc}(\R^N)$,  with $ p = \max \{m, \theta, 2s+\gamma\}$. 

\smallskip

Assume further then $ \theta = 2s + \gamma$. 

\item[i)] If $ \, \inf_{\R^N} (\lambda - \beta) >0 $, then $ u=v=0$.

\item[ii)] If $\lambda \ge \beta$, $ \alpha \ge \alpha_0 >0$ and $ m>1$, then $ u=v=0$.

\item[iii)] If $ m=\theta$ and $ \, \inf_{\R^N} (\alpha + \lambda - \beta) >0 $, then $ u=v=0$.

\item[iv)] If $\lambda = \beta$, $ \alpha  =0 $, then $u=v$ and $u$ is a harmonic function. In particular,  if  either $u$ or $v$ is  bounded  on one side, then $ u=v=const.$

\end{theorem}

\begin{proof} 
By the previous theorem we have $ u=v$. Hence 

\beq \label{equazunica} 
\Delta u = \alpha \vert u \vert^{m-1} u + (\lambda- \beta) \vert u \vert^{\theta-1} u \qquad {\rm in} \quad {\cal D}^{'}(\R^N)
\eeq
and by Kato's inequality (once again) we see that

\beq \label{equazunicaKato} 
\Delta u^+ \ge \alpha (u^+)^{m} + (\lambda- \beta) (u^+)^{\theta} \qquad {\rm in} \quad {\cal D}^{'}(\R^N).
\eeq

If i) (or ii) or iii)) is in force, there are $\epsilon>0$ and $ \eta>1$ such that 

\beq \label{eq3} 
\Delta u^+ \ge \epsilon  (u^+)^{\eta}  \qquad {\rm in} \quad {\cal D}^{'}(\R^N).
\eeq

Thus $u^+ =0$ and then $ u \le 0 $ on $ \R^N$. On the other hand, also $-u$ is a solution of \eqref{equazunica}, hence $ u^- =0$ and the desired conclusion follows.

If iv) holds true, $u$ and $v$ are harmonic functions. Hence $u=v=const.$ by the classical Liouville Theorem.

\end{proof}
\medskip

Some remarks are in order :

\medskip

\begin{remark}  \label{rem1}
{\em

\item[1)] The assumption : $ \lambda = \lambda(x) \ge \lambda_0>0$  is necessary,  as it easily seen by choosing $ m=\alpha=1 $ and $ \lambda =\beta =0$ in \eqref{Sistema}.  In this case system \eqref{Sistema} reduces to 

\beq\label{Sistemasempli}
\begin{cases}

- \Delta u  + u  = 0 & \text{in $\R^N$}\\
- \Delta v  +  v  = 0 & \text{in $\R^N$}
\end{cases}
\eeq
which admits positive, non-trivial and non-symmetric solutions.  For instance $ u(x) = e^{x_1}$ and $ v(x) = e^{- x_1} + 2 e^{x_1}$, where $ x = (x_1,...,x_N) \in \R^N$. 

\item[2)] The assumptions : '' $ \lambda - \beta \ge0 $ " in the above Thereom \ref{teo2} are essentially necessary. Indeed, when this condition is not satisfied, there are non constant solutions (cfr. for instance the following Theorems \ref{teoCLASS1}, \ref{teoCLASS2} and the existence results in \cite{LinWei}). 

 }
\end{remark}

\begin{remark} \label{rem2}
{\em Theorem \ref{teo1} recovers and significantly improves some previous results demonstrated in \cite{QS1} and \cite{Ma-Zhao}. Indeed, by choosing $ \alpha =0$, $\theta =3$, $s=1$, $ \gamma =1$ in \eqref{Sistema}, we recover the cubic system $(1.6)$ in \cite{QS1} 

\beq\label{SistemaBE}
\begin{cases}

- \Delta u  = uv^2 - \lambda u^3 & \text{in $\R^N$}\\
- \Delta v  = vu^2 - \lambda v^3 & \text{in $\R^N$}\\
\quad u\ge0, \, v \ge0 & \text{on $\R^N$}
\end{cases}
\eeq
which appears in Bose-Einstein condensation, and by choosing $ \alpha =1$, $\theta =2r+1$, $s=r$, $ \gamma =1$  in \eqref{Sistema}, we recover system $(1.8)$ in \cite{QS1}  and system $(6)$ in \cite{Ma-Zhao}

\beq\label{SistemaNOP}
\begin{cases}

- \Delta u  + u^m  = - \lambda u^{2r +1} + \beta u^r v^{r+1} & \text{in $\R^N$}\\
- \Delta v  + v^m  =  - \lambda v^{2r +1} + \beta v^r u^{r+1}  & \text{in $\R^N$}\\
\quad u\ge0, \, v \ge0 & \text{on $\R^N$}
\end{cases}
\eeq
arising in nonlinear optics. 

For instance, we recover and extend Theorem 1.6 of \cite{QS1} and Theorem 2 of \cite{Ma-Zhao} (cfr. also Remark 1.8 of \cite{QS1}), since $(u,v) $ is merely a distributional solution, possibly sign-changing and no further assumption is made on the solution $(u,v)$.  In particular, $(u,v)$ need not to be neither a ground  state nor a classical positive decaying solution. Furthermore, as far as system \eqref{SistemaBE} is concerned, we do not have any restriction about the parameter $\lambda >0$.    

Moreover, if we restrict our attention to non-negative solutions, we can further extend the above mentioned results to obtain precise classification results for models naturally arising in physical applications. More precisely we have :

 } 
\end{remark}

\medskip

\begin{theorem} {\bf  (of classification I) }\label{teoCLASS1}
Let $ N\ge 1 $ and assume $s\ge0$, $\gamma \ge1$, $\beta \ge 0 $, $ \lambda >0$. Let $(u,v)$ be a non-negative distributional solution of

\beq\label{SistemaBEGene}
\begin{cases}

- \Delta u  = \beta u^s v^{s+\gamma} - \lambda u^{2s +\gamma} & \text{in $\R^N$}\\
- \Delta v  =  \beta v^s u^{s+\gamma} - \lambda v^{2s +\gamma} & \text{in $\R^N$}
\end{cases}
\eeq
such that  $ u, v \in L^{2s+\gamma}_{loc}(\R^N)$.  

\medskip

\noindent Then $ u=v$.

\medskip

\noindent Furthermore,  

\item[i)] if $ \, \lambda > \beta $, then $ u=v=0$.

\item[ii)] If $\, \lambda = \beta$, then $ u=v=const.$

\item[iii)] If $ \, \lambda < \beta $ and 

\beq \label{expJames} 1 <  2s+\gamma \le 
\begin{cases}
+\infty & \text{if}  \quad N \le 2,\\\\
\frac{N}{N-2} & \text{if}  \quad N \ge 3,
\end{cases}
\eeq
then $u = v = 0$.

\item[iv)] If $ \, \lambda < \beta$,  $N \ge 3$, $ u,v \in H^1_{loc}(\R^N) \cap L^{2s+\gamma}_{loc}(\R^N)$  and 

\beq \label{expSob} \frac{N}{N-2} < 2s+\gamma < \frac{N+2}{N-2}
\eeq
then $u = v = 0 $.

\item[v)] If $ \, \lambda < \beta$,  $N \ge 3$, $ u,v \in H^1_{loc}(\R^N) \cap L^{2s+\gamma}_{loc}(\R^N)$  and 

\beq \label{expSobCrit} 2s+\gamma = \frac{N+2}{N-2}
\eeq
then either $ u= v=0$ or

\beq \label{Caf-Gidas-Spr}
u = v = c(N, \lambda, \beta) \Big [ \frac{\eta}{\vert x - x_0\vert^2 + \eta^2}  \Big ]^{\frac{N-2}{2}}
\eeq
for some $ \eta>0$, $x_0 \in \R^N$ and $ c(N, \lambda, \beta)>0$.

\end{theorem}

\medskip

\begin{remark} 
{\em

\item [1)] Note that system \eqref{SistemaBE} is obtained by setting  $\beta =1$, $s= \gamma =1$ in \eqref{SistemaBEGene}. 

\item [2)] Theorem \ref{teoCLASS1} provides a complete classification in dimension $ N\le2$ (no restriction is made on the parameters $s, \gamma, \beta$ and $ \lambda$).  It also provides a complete classification for the cubic system \eqref{SistemaBE} in dimension $ N \le 4$ (cfr. also Remark \ref {rem2}). 

\item [3)]  The assumption $ u,v \in H^1_{loc}(\R^N)$ is necessary in {\it (iv)} and {\it (v)}. Indeed, for instance, for all $s\ge 0$ and $\gamma\ge 1$ such that $2s+\gamma >  \frac{N}{N-2}$, there is a singular radial positive solution $(u,v) $ with $ u = v = c(N,s,\gamma, \lambda,\beta) \vert x \vert^{-\frac{2}{2s +\gamma -1}}$ and where $c(N,s,\gamma, \lambda, \beta)>0$ is an explicit constant.

}
\end{remark}

\medskip

\begin{proof}[Proof of Theorem \ref {teoCLASS1}] 
By Theorem \ref{teo1} we have $u=v$. Items {\it i)} and {\it ii)} follow directly from {\it i)} and {\it iv)} of Theorem \ref{teo2}. 
To proceed, we observe that system \eqref{SistemaBEGene} reduces to the equation 

\beq \label{eqLE} 
 - \Delta u  = (\beta-\lambda) u^{2s +\gamma} \qquad {\rm in} \quad {\cal D}^{'} (\R^N).
\eeq

By a standard density argument, we can use test functions of class $ C^2_c$ in the above equation \eqref{eqLE}. Thus, the desired result follows, for instance, from Theorem 2.1 of \cite{MP}. 

%To prove {\it iii)} we set $ p = 2s + \gamma>1$ and fix $m > \max \{2,p'\}.$ Using equation \eqref{eqLE} we get, for every $ \varphi \in C^{\infty}_c(\R^N)$, $\varphi \ge0$, 

%\[
%\int u^p \varphi^m = - (\beta - \lambda) \int u \Delta( \varphi^m) 
%\]
%\[
%= - (\beta - \lambda) \int u(m \varphi^{m-1} \Delta \varphi + m(m-1) \varphi^{m-2} \vert \nabla \varphi \vert^2 ) \le - m (\beta - \lambda) \int u \varphi^{m-1} \Delta \varphi
%\]
%\[
%\le  \int  \frac{u^p}{p} \varphi^m + [m (\beta - \lambda)]^{p'} \int  \frac{\varphi^{m-p'}}{p'} \vert \Delta \varphi %\vert^{p'}
%\]
%which leads to
%\beq \label{cap}
%\int {u^p}\varphi^m \le [m (\beta - \lambda)]^{p'} \int  {\varphi^{m-p'}}\vert \Delta \varphi \vert^{p'}.
%\eeq

%Now we conclude by making use of the standard cut-off functions $\f_R(x):=\f(|x|/R)$, $R>0$, where $\f: [0,+\infty) \to \R$ is a function of class $\mathcal{C}^\infty$ such that
%\[
%\begin{cases}
%\f(t)=1 & t \in [0,1], \\
%\f(t)=0 & t \in [2,+\infty),\\
%0 \leq \f(t) \leq 1 & t \in (1,2).
%\end{cases} 
%\]

%Inserting $\f_R$ into \eqref{cap} we obtain 

%\beq \label{cap2}
%\int_{B_R(0)} {u^p}\varphi^m \le C(N,p')[m (\beta - \lambda)]^{p'} R^{N-2p'},
%\eeq
%and the desired conclusion follows by observing that, in view of \eqref{expJames}, the R-H-S of \eqref{cap2} tends to %zero as $R$ goes to $+\infty$.  Note that the above argument works only assuming that $u$ is a non-negative distributional supersolution of \eqref{eqLE}. 

%\medskip

When {\it iv}) (or {\it v})) is in force, it is well-known that $u$ is a classical solution (i.e. of class $C^2$) of the equation \eqref{eqLE}. Hence, the claims follow immediately from the celebrated results of Gidas and Spruck \cite{GS1, GS2} and of Caffarelli, Gidas and Spruck \cite{CGS}. \end{proof} 

\medskip

Now we turn our attention to the system \eqref{SistemaNOP} and we prove the following

\medskip

\begin{theorem} {\bf  (of classification II) }\label{teoCLASS2}
Let $ N\ge 1 $ and assume $m>0$, $r>0$, $\beta \ge 0 $, $ \lambda >0$. Let $(u,v)$ be a non-negative distributional solution of 

\beq\label{SistemaNOP2}
\begin{cases}

- \Delta u  + u^m  = - \lambda u^{2r +1} + \beta u^r v^{r+1} & \text{in $\R^N$}\\
- \Delta v  + v^m  =  - \lambda v^{2r +1} + \beta v^r u^{r+1}  & \text{in $\R^N$}
\end{cases}
\eeq
such that  $u, v \in L^p_{loc}(\R^N)$,  with $ p = \max \{m, 2r+1\}$.  

\medskip

\noindent Then $ u=v$.

\medskip

\noindent Furthermore,  

\item[i)] if $ \, \lambda > \beta $, then $ u=v=0$.

\item[ii)] If $\, \lambda = \beta$ and $ m > 1$,  then $ u=v=0.$

\item[iii)] If $ \, \lambda < \beta $ and $ 2r+1 < m$, 

\noindent then $(u,v)$ is a smooth, bounded and classical solution of \eqref{SistemaNOP2}, it satisfies the following universal and sharp $L^{\infty}$-bound 

\beq \label{bound}
 \Vert u \Vert_{\infty} =\Vert v \Vert_{\infty} \le (\beta - \lambda)^{\frac{1}{m-(2r+1)}}.
\eeq

\medskip

\noindent Moreover, if either $ N\le 2$ or, $ N\ge 3 $ and $ 2r + 1 \le \frac{N+2}{N-2} $,  
then either $u = v = 0$ or $  u = v = (\beta - \lambda)^{\frac{1}{m-(2r+1)}}$. 

\smallskip

\item[iv)] If $ \, \lambda < \beta$ and $ 2r+1 =m$, we have :

\item[1)] if $ 0< \beta - \lambda <1$, then $ u=v=0$.
\item[2)] if $ \beta - \lambda =1$, then $ u=v=const.$
\item[3)] if $ \beta - \lambda >1$ and 

\beq \label{expJames} 2r+1 \le
\begin{cases}
+\infty & \text{if}  \quad N \le 2,\\\\
\frac{N}{N-2} & \text{if}  \quad N \ge 3,
\end{cases}
\eeq
then $u = v = 0$.  

\item[4)] if $ \beta - \lambda >1$, $N \ge 3$, $ u,v \in H^1_{loc}(\R^N) \cap L^{2r+1}_{loc}(\R^N)$  and 

\beq \label{expJames} 2r+1 <
\begin{cases}
+\infty & \text{if}  \quad N \le 2,\\\\
\frac{N+2}{N-2} & \text{if}  \quad N \ge 3,
\end{cases}
\eeq
then $u = v = 0$.

\item[5)] if $ \beta - \lambda >1$, $N \ge 3$, $ u,v \in H^1_{loc}(\R^N) \cap L^{2r+1}_{loc}(\R^N)$  and 

\beq \label{expSobCrit2} 2r+1 = \frac{N+2}{N-2}
\eeq
then either $ u= v=0$ or

\beq \label{Caf-Gidas-Spr2}
u = v = c(N, \lambda, \beta) \Big [ \frac{\eta}{\vert x - x_0\vert^2 + \eta^2}  \Big ]^{\frac{N-2}{2}}
\eeq
for some $ \eta>0$, $x_0 \in \R^N$ and $ c(N, \lambda, \beta)>0$. 
 
\item[v)] If $ \, \lambda < \beta$ and $ 2r+1 >m $ we have :

\item[1)] when $ m \ge1$, $u,v \in C^0(\R^N)$ and either $u$ or $v$ tends to zero uniformly, as $ \vert x \vert \to +\infty$, then either $u=v=0$ or

\noindent  $u=v>0$ everywhere on $\R^N$ and $u$ is necessarily radially symmetric and strictly radially decreasing, i.e., $ u(x) =v(x) = w(\vert x - x_0 \vert)$, for some $ x_0 \in \R^N$ and some positive function 
$w$ such that $ w'(0) = 0$ and $ w^{'} (r) <0$ for $r>0$. Moreover, the profile $w$ is unique. 

\item[2)] when $ m < 1$, $u,v \in C^0(\R^N)$ and either $u$ or $v$ tends to zero uniformly, as $ \vert x \vert \to +\infty$, then either $u=v=0$ or, $u=v$ is compactly supported, $u$ has open support (i.e. the set $\{ \, x \in \R^N : u(x)>0 \, \} $) on a finite number of open balls in $\R^N$, on each of which it is radially symmetric about the center of the ball and its profile is unique. 

\smallskip

Furthermore, when $ N \ge 2$,  system \eqref{SistemaNOP2} admits a non-constant solution $(u,v)$ such that $u$ or $v$ tends to zero uniformly, as $ \vert x \vert \to +\infty$, if and only if  

%$ 2r+1 < \frac{N+2}{N-2} $. 

%Finally, if $m=1$ and $ N \ge1$, system \eqref{SistemaNOP2} admits a non-constant solution $(u,v)$, such that $u$ or $v$ tends to zero uniformly, as $ \vert x \vert \to +\infty$, if and only if
\beq \label{expGS} 2r+1 <
\begin{cases}
+\infty & \text{if}  \quad N = 2,\\\\
\frac{N+2}{N-2} & \text{if}  \quad N \ge 3.
\end{cases}
\eeq

\end{theorem}

\medskip

\begin{remark} {\em

The situation is more complicated when $ m < 1 $. Indeed,

\item [1)]  in view of conclusion $2)$ of item $v)$,  system \eqref{SistemaNOP2} admits non-negative, non-constant, compactly supported classical solutions. This also shows the importance to consider non-negative solutions.

\item [2)]  For every $ x_0 \in \R$ set

\beq \label{CompSupp} w(t) : = 
\begin{cases}
0 & \text{if} \quad t \le 0,\\\\
[(\frac{2}{1-m}) (\frac{2}{1-m} -1)]^{-\frac{1}{1-m}} \, t^{\frac{2}{1-m}}& \text{if}  \quad t \ge 0,
\end{cases}
\eeq
and $u_{x_0}(x) : = w(x_1 - x_0,....,x_N) $.  The couple $(u_{x_0}, u_{x_0})$ is a non-negative, non-constant, classical solution of \eqref{SistemaNOP2} with $ \lambda = \beta  $ and $ 0 < m <1$.  Combining this example  with the example of Remark \ref{rem1}, we see that the conclusion of item $ii)$ is sharp. 

}
\end{remark}

\medskip

\begin{proof}[Proof of Theorem \ref {teoCLASS2}] 
By Theorem \ref{teo1} we have $u=v$. Items {\it i)} and {\it ii)} follow directly from {\it i)} and {\it ii)} of Theorem \ref{teo2}.  Since $u=v$, system \eqref{SistemaNOP2} reduces to the equation

\beq \label{eqEqNOP} 
 - \Delta u  + u^m = (\beta-\lambda) u^{2r+1} \qquad {\rm in} \quad {\cal D}^{'} (\R^N).
\eeq
with $ \beta - \lambda>0$. 

To prove $iii)$, for every $ \varepsilon >0$ we set $u_{\varepsilon} := u - (\beta - \lambda)^{\frac{1}{m-(2r+1)}} -\varepsilon$ and apply Kato's inequality to \eqref{eqEqNOP} to get 

\beq 
\Delta u_{\varepsilon}^+ \ge \Delta u 1_{\{ u_{\varepsilon} >0 \}} = 
\eeq
\beq
u^{2r+1} [u^{m-(2r+1)} - (\beta - \lambda)] 1_{\{ u_{\varepsilon} >0 \}} \ge
c [u_{\varepsilon}^+]^{2r+1} \qquad {\rm in} \quad {\cal D}^{'} (\R^N), 
\eeq
where $c$ is a positive constant depending on $ \varepsilon, m, r, \beta $ and $\lambda$.   From the latter we infer $ u_{\varepsilon}^+ \le 0$ on $\R^N$ and thus $ u \le (\beta - \lambda)^{\frac{1}{m-(2r+1)}} $ by letting $ \varepsilon \to 0$.  This gives the bound \eqref{bound}, whose sharpness follows by noticing that $u=v=(\beta - \lambda)^{\frac{1}{m-(2r+1)}}$ is a solution of \eqref{SistemaNOP2}. The smoothness of $(u,v)$ immediately follows from standard elliptic regularity, since $(u,v) \in L^{\infty}$. 
In view of \eqref{bound} and of \eqref{eqEqNOP} we see that $u$ is a smooth positive superharmonic function. Thus $u$ must be constant when $N\le2$ and the only possibilities are $u = 0$ or $  u = (\beta - \lambda)^{\frac{1}{m-(2r+1)}}$ (since $u$ solves \eqref{eqEqNOP}). To treat the case $ N\ge3$ we need to use, in an essential way, the fact that we proved that $u$ is smooth and satisfies the bound \eqref{bound}. Indeed, in view of those properties of $u$,  we can invoke Theorem 2.4 of \cite{DD}, when $ 2r+1 < \frac{N+2}{N-2} $, and Theorem 3 of \cite{B}, when $ 2r+1 = \frac{N+2}{N-2} $, to obtain the desired conclusion.

When $iv)$ is in force, system \eqref{SistemaNOP2} boils down to 

\beq \label{} 
 - \Delta u  = (\beta-\lambda -1) u^{2r+1} \qquad {\rm in} \quad {\cal D}^{'} (\R^N),
\eeq
and the claims follows as in the proof of item $iii)$ of Theorem \ref{teoCLASS1}. 

Proof of $v)$ : $1)$ by the strong maximum principle, either $u=v=0$ or $ u=v>0$ on $\R^N$.  In the latter case $u$ is radially symmetric and strictly radially decreasing by the well-known results of \cite{GNN, SZsim}. Uniqueness of the profile $w$ follows from \cite{PuS, STang}.  

The claims of part $2)$ follows by applying the results of \cite{CEF, SZsim} and \cite{PuS, STang}.

The results for $ N \ge2 $ follow from \cite{CEF, STang} and \cite{GST} (cfr. also the references therein). 
 
%For $m=1$, the existence goes back to \cite{Stra}, the uniqueness was proved in \cite{Kw} (cfr. also the references therein). The non-existence for critical and supercritical values of $2r+1$ is a well-known consequence of Pohozeav  identity (note that, when $m=1$,  $u$ belongs to $H^1(\R^N)$, due to its exponential decay at infinity). 

\end{proof}

\section{More general results}

The method used to prove the above results also applies to more general systems and with nonlinearities which are not necessarily of polynomial type. To this end we need to recall the well-known {\it Keller-Osserman condition} \cite{Keller,Oss}.  

A non-decreasing function $f  \in C^0([0, +\infty) , 
[0, +\infty))$ is said to satisfy the {\it Keller-Osserman condition} if

\beq\label{K-Ocond}
\begin{cases}

f(0) = 0 \text{},\\

f(t)>0, \quad  \text {if} \quad  t>0,\\

\int_{}^{+\infty} \Big[ \int_{0}^s f(t) dt \Big]^{- {1\over2}} ds < +\infty. 
\end{cases}
\eeq

A typical example of function satisfying the above condition \eqref{K-Ocond} is $ f(t) = t^q$, $ q>1$. Also $ f(t) = t\log^{\delta}(t+1)$, $\delta>2$,  satisfies \eqref{K-Ocond}, while $f(t) =t $ does not fullfill \eqref{K-Ocond}. 

\medskip

\begin{theorem} \label{teo3}
Assume $ N\ge 1 $ and let $(u,v)$ be a distributional solution of 
\beq\label{Sistema2}
\begin{cases}

- \Delta u  = h(x,u,v)   & \text{in $\R^N$}\\
- \Delta v  = h(x,v,u)  & \text{in $\R^N$}
\end{cases}
\eeq

\noindent where $ h : \R^N \times \R^2 \rightarrow \R $ is continuous and satisfies 

 \beq \label{monot}   
 h(x,v,u)-h(x,u,v) \ge f(u-v) \qquad \forall \,  u \ge v, \quad \forall x \in \R^N
\eeq
and $ f $ is a convex function fulfilling  the  Keller-Osserman condition. 

\smallskip

\noindent If $u,v \in L^1_{loc} (\R^N)$ and $ h(\cdot,u,v), h(\cdot, v,u) \in L^1_{loc}(\R^N)$,  then $ u=v$.

\begin{proof}

Set $\psi = u-v$. The assumptions on $u$ and $v$ imply that both $\psi$ and $ \Delta \psi$ belong to $ L^1_{loc}(\R^N)$. Hence Kato's inequality yields

\beq \nonumber
\Delta (\psi^+) \ge (h(x,v,u)-h(x,u,v) ) 1_{\{ u-v>0\}} 
\eeq
\beq  \label{differenzaeqKatof} 
\ge  f(u-v) 1_{\{ u-v>0\}}  = f(\psi^+) \qquad {\rm in} \quad {\cal D}^{'}(\R^N). 
\eeq
Thus we can apply Theorem 4.7 of \cite{Far} (where we have set $ f(t) =0$ if $ t \le 0$) to get that  $ \psi^+ = 0$. To conclude we proceed as in the proof of Theorem \ref{teo1}. 
\end{proof}

\end{theorem}

The above theorem is not true if the Keller-Osserman conditon is not satisfied, as witness the example given by system \eqref{Sistemasempli}, which is of the form \eqref{Sistema2} with $ h(x,u,v) = -u$ and satisfies \eqref{monot}  with $f(t) = t$. 

\smallskip

Nevertheless not all is lost, since we have the following:

\medskip

\begin{theorem} \label{teo4}
Assume $ N\ge 1 $ and let $(u,v)$ be a distributional solution of 
\beq\label{Sistema2a}
\begin{cases}

- \Delta u  = h(x,u,v)   & \text{in $\R^N$}\\
- \Delta v  = h(x,v,u)  & \text{in $\R^N$}
\end{cases}
\eeq

\noindent where $ h : \R^N \times \R^2 \rightarrow \R $ is continuous and satisfies 

\beq 
 h(x,v,u)-h(x,u,v) \ge \nu (u-v) \qquad \forall \,  u \ge v, \quad \forall x \in \R^N,
\eeq
for some  constant $\nu>0$.

\smallskip

\noindent If $u,v \in L^1_{loc} (\R^N)$ and $ h(\cdot,u,v), h(\cdot, v,u) \in L^1_{loc}(\R^N)$,  then $ u=v$, whenever $u$ and $v$ have at most polynomial growth at infinity.
\end{theorem}

\begin{proof} Set $ \psi = u-v$. As in the proof of Theorem \ref{teo3} we obtain 
\beq \nonumber
\Delta (\psi^+) \ge (h(x,v,u)-h(x,u,v) ) 1_{\{ u-v>0\}} 
\eeq
\beq  \label{differenzaeqKatononKO} 
\ge  \nu (u-v) 1_{\{ u-v>0\}}  = \nu \psi^+ \qquad {\rm in} \quad {\cal D}^{'}(\R^N). 
\eeq

We consider a $\mathcal{C}^\infty$ function $\f: [0,+\infty) \to \R$ such that
\[
\begin{cases}
\f(t)=1 & t \in [0,1], \\
\f(t)=0 & t \in [2,+\infty),\\
0 \leq \f(t) \leq 1 & t \in (1,2),
\end{cases} 
\]
and we set, for every $R>0$ and every $x \in \R^N$, $\, \f_R(x):=\f(|x|/R)$. 

Using  the cut-off functions $\f_R$ as test functions in \eqref{differenzaeqKatononKO}, and recalling that  $\psi^+$ has at most polynomial growth at infinity,  we have for any $R > 1$

\[
\int_{B_R} \psi^+\leq \frac{C}{\nu R^2}  \int_{B_{2R}} \psi^+ \leq C' R^{N + k - 2}
\]
for some $k\ge 0$, $C'>0$ independent of $R$.

Iterating the latter a finite number of times, we immediately obtain that

\[
\forall \, R>1 \qquad \int_{B_R} \psi^+ \leq C'' R^{-m}
\]
for some $m > 0$, $C''>0$ independent of $R$. This leads to  $\int_{\R^N} \psi^+ = 0,$ which in turn yields $ u \le v$ a.e. on $\R^N.$ Finally, exchanging the role of $u$ and $v$ we obtain the desired conclusion $u = v$.\end{proof}

\medskip

We conclude this section by noticing that, classification results similar to those of Theorem \ref{teoCLASS1} and/or Theorem \ref{teoCLASS2}, can also be established for solutions to the system \eqref{Sistema2},  under suitable assumptions on the function $h$. Nevertheless,  we do not want to stress on this point. 

\vspace{0.5cm}

\noindent \textbf{Acknowledgements: } The author thanks P. Quittner and P. Souplet for a careful reading of a first version of this article.

\noindent The author is supported by the ERC grant EPSILON ({\it Elliptic Pde's and Symmetry of Interfaces and Layers for Odd Nonlinearities}).

\end{document}